\documentclass[12pt, a4]{amsart}

\usepackage{amsmath,amssymb,cite,mathrsfs,graphicx}

\newtheorem{thm}{Theorem}
\newtheorem{prop}{Proposition}

\newtheorem{obs}{Observation}

\theoremstyle{definition}

\theoremstyle{remark}
\newtheorem{rem}{Remark}

\numberwithin{equation}{section}

\title[On the zeros of the Euler double zeta-function II]
{Numerical computations on the zeros of the Euler double zeta-function II} 
\author{Kohji Matsumoto, Mayumi Sh{\=o}ji}

\begin{document}
\maketitle\thispagestyle{empty}

\begin{abstract}
We study the behavior of zero-sets of the double zeta-function
$\zeta_2(s_1,s_2)$ (and also of more general multiple zeta-function
$\zeta_r(s_1,\ldots,s_r)$).    
In our former paper 
we studied the case
$s_1=s_2$, but in the present paper we consider the more general two variable
situation.   We carry out
numerical computations in order to trace the behavior of zero-sets
of $\zeta_2(s_1,s_2)$.
We observe that some
zero-sets approach the points $(s_1,s_2)$ with $s_2=0$, while other zero-sets approach the
points $(s_1,s_2)$, where $s_2$ are solutions of $\zeta(s_2)=1$.
In the former case, when $s_2$ tends to $0$, we observe that $\Im s_1$ becomes close
to the imaginary part of a non-trivial zero of the Riemann zeta-function.
In the latter case we give a theoretical proof, in the general $r$-fold setting.
\end{abstract}


\renewcommand{\thefootnote}{}
\footnotetext{\hspace*{-.51cm}AMS 1991 subject classification: 
Primary: 11M32, Secondly: 11M35} 

\section{Introduction}\label{sec-1}

The present paper is a continuation of the authors' previous article \cite{MatSho14}
on the study of the zeros of the Euler double zeta-function
\begin{align}\label{1-1}
\zeta_2(s_1,s_2)=\sum_{n_1=1}^{\infty}\sum_{n_2=1}^{\infty}
\frac{1}{n_1^{s_1}(n_1+n_2)^{s_2}},
\end{align}
where $s_1,s_2$ are complex variables.    This double series is convergent absolutely in 
the region defined by $\Re s_1+\Re s_2>2$ and $\Re s_2>1$, and can be continued
meromorphically to the whole complex space $\mathbb{C}^2$ (see \cite{Mat02}).

This is the case $r=2$ of the more general Euler-Zagier $r$-fold sum
\begin{align}\label{1-2}
\zeta_r(s_1,\ldots,s_r)=\sum_{1\leq n_1<\cdots<n_r}\frac{1}{n_1^{s_1}\cdots
n_r^{s_r}},
\end{align}
which has been investigated quite extensively from various aspects.

In order to understand the analytic properties of \eqref{1-2}, it is very important to
study the behavior of its zeros and singularities.
It is already known that possible singularities of \eqref{1-2} are located only on
the set $S_r$, which is the union of hyperplanes in
$\mathbb{C}^r$ defined by 
\begin{align*}
&A_r^{(j,n)}:=\{(s_1,\ldots,s_r)\in\mathbb{C}^r\;|\;s_{r-j+1}+\cdots+s_r=j-n\} \\
&\qquad\qquad\qquad\qquad(n=0,1,2,\ldots,\; 2\leq j\leq r),\\
&A_r^{(1)}:=\{(s_1,\ldots,s_r)\in\mathbb{C}^r\;|\;s_r=1\}
\end{align*}
(see \cite{AET}, \cite{MatJNT}).
However, the distribution of the zeros of \eqref{1-2} has not been, except for the 
classical case of $r=1$ (that is the case of the Riemann zeta-function $\zeta(s)$), 
studied in detail.
    The aim of the present series of papers is to
study the behavior of the zeros of \eqref{1-1}, the simplest case (except for the case
$r=1$), from the viewpoint of numerical computations.

In \cite{MatSho14}, we considered the situation when $s_1=s_2(=s)$.
Then $\zeta_2(s,s)$ is a function of one variable, so we can study the distribution of
the zeros of $\zeta_2(s,s)$ in a way analogous to the case of $\zeta(s)$.
Unlike the case of $\zeta(s)$, the function $\zeta_2(s,s)$ does not satisfy the
analogue of the Riemann hypothesis.   We found a lot of zeros in the strip
$0\leq \Re s\leq 1$ off the line $\Re s=1/2$, or even outside that strip (see \cite[Observation 1]{MatSho14} and 
\cite[Figure 1]{MatSho14}).    We pointed out that the distribution of those zeros is 
similar, not to the case of $\zeta(s)$, but rather, to the case of Hurwitz zeta-functions.


In the present paper, we consider the general situation, when $s_1$ and $s_2$ are
moving independently.    The zeros of $\zeta_2(s_1,s_2)$, as a function of two variables, 
are not isolated points.    They form analytic sets in $\mathbb{C}^2$, which we call 
{\it zero-sets}.   We carry out
numerical computations in order to trace the behavior of such zero-sets.
Mainly we consider the asymptotic situation of those zero-sets
when, for example, $|s_1-s_2|$ becomes large, or $s_2\to 0$, or
$|s_1| \to \infty$.
We observe that some
zero-sets approach the points $(s_1,s_2)$ with $s_2=0$, while other zero-sets approach the
points $(s_1,s_2)$ where $s_2$ are solutions of $\zeta(s_2)=1$.

It seems that, in general, to give theoretical proofs for the observations
described in the present paper is a very difficult task.     
However, at least, we can construct a proof of the
fact that some zero-sets approach the points satisfying $\zeta(s_2)=1$.
In fact, we will prove a more general result for the $r$-fold sum \eqref{1-2}
in the last section.

{\bf Acknowledgment}.
The authors express their sincere gratitude to the referees for their useful comments
and suggestions.

\section{The behavior of zero-sets}

Our basic strategy is to begin with the zeros of $\zeta_2(s,s)$ discovered in
\cite{MatSho14}, and study the behavior of zero-sets of $\zeta_2(s_1,s_2)$
around those zeros.  
  
Let us begin with \cite[Figure 1]{MatSho14}, which we reproduce here as Figure \ref{X1-1}, 
on which a lot of non-real $s$, for which
$\zeta_2(s,s)=0$ holds, are dotted.    The values of some of which are given in the list
written on \cite[pp.308-309]{MatSho14}, whose order is according to the
magnitude of the imaginary parts of them.    The first four of them are
\begin{align*}
&(0.27672860) + i(8.39755368),\\
&(-0.18995147) + i(12.30422130),\\
&(0.06443907) + i(15.02312694),\\
&(-0.53767831) + i(17.58063303).
\end{align*}
Denote those values by
$a_1,a_2,a_3,\ldots$ and so on.
(See Remark \ref{rem-1} below for some corrections of the data.)

\begin{figure}
\begin{center} \leavevmode
\includegraphics[width=0.42\textwidth]{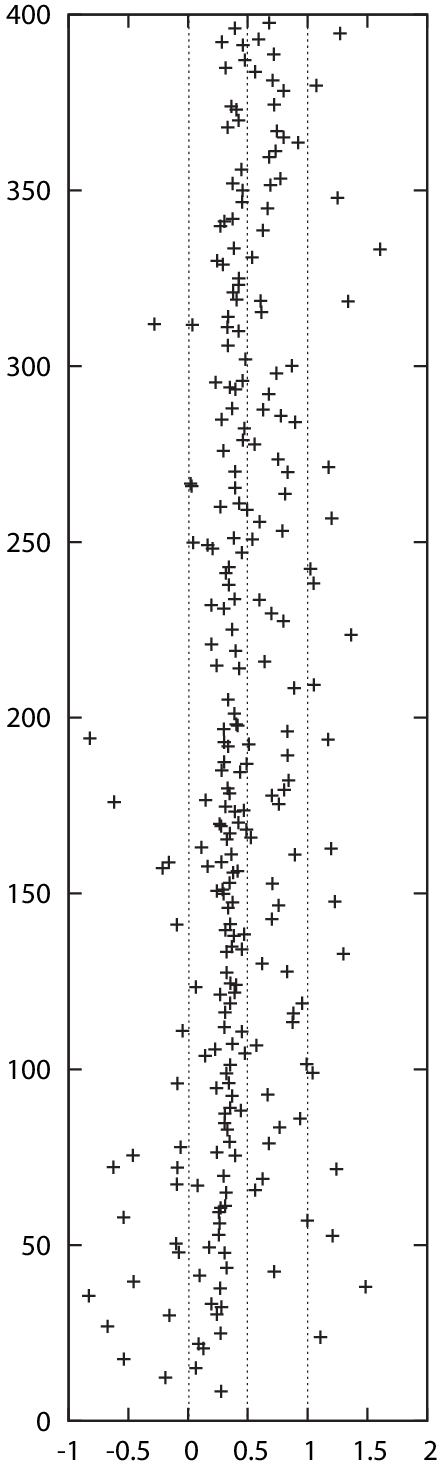} \hskip 1cm
\includegraphics[width=0.42\textwidth]{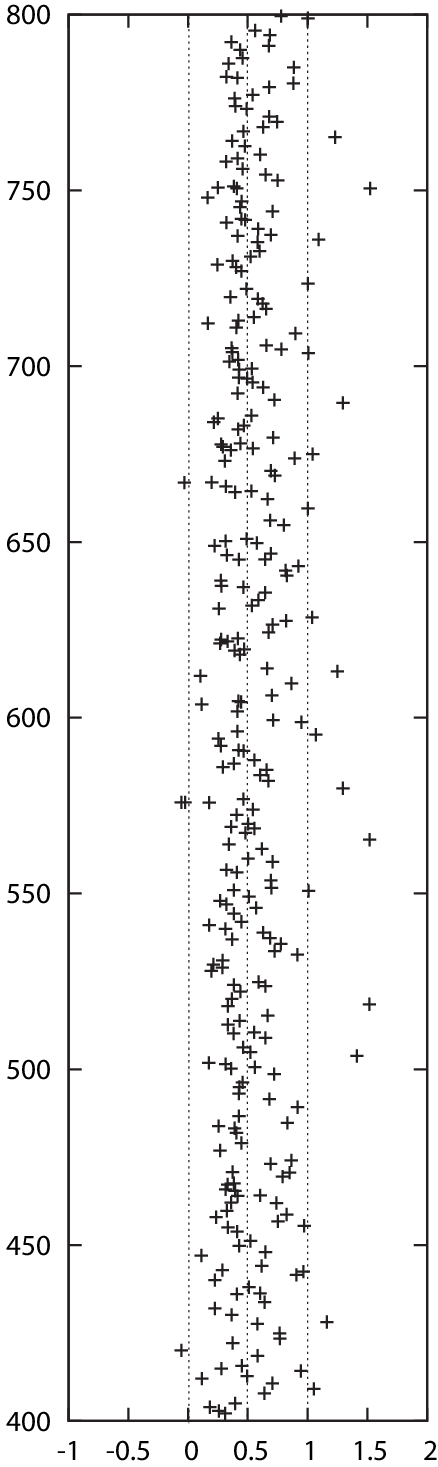}
\end{center}
\caption{The distribution of zeros of $\zeta_2(\sigma+it,\sigma+it)$ for $-1 \leq \sigma \leq 2$ and $0 \leq t \leq 800$. The horizontal axis represents $\sigma$ and the vertical axis does $t$.}
\label{X1-1}
\end{figure}

Now consider the two-variable function $\zeta_2(s_1,s_2)$.
Since $\zeta_2(s_1,s_2)$ cannot have any isolated zero point, 
the points $(a_i,a_i)\in\mathbb{C}^2$ are to be
intersections of some zero-sets and the hyperplane $s_1=s_2$.    Our first aim is
to investigate the behavior of zero-sets near the points $(a_i,a_i)$.

Let $\delta$ be a positive number.    We search for the zeros of $\zeta_2(s_1,s_2)$
around the point $(a_i,a_i)$ under the condition $|s_1-s_2|=\delta$.    
The method of computations is based on the Euler-Maclaurin
formula, explained in \cite[Section 4]{MatSho14}.     
(It is to be noted that the simple method using the harmonic product formula
\cite[(2.1), (2.2)]{MatSho14} cannot be applied to the present situation where
$s_1\neq s_2$.) 
Figure \ref{Fig1-1} describes the loci of the absolute values of zeros 
satisfying $|s_1-s_2|=\delta$ around $(a_i,a_i)$ 
($5\leq i\leq 9$) for various values of $\delta$.
We quote the values of those $a_i$s from \cite{MatSho14}:
\begin{align*}
&a_5=(0.12844956) + i(20.59707674),\\
&a_6=(0.08804454) + i(21.93232180),\\
&a_7=(1.10778631) + i(23.79708697),\\
&a_8=(0.27268471) + i(24.93425087),\\
&a_9=(-0.67413685) + i(26.88584448).
\end{align*}

The left one of Figure \ref{Fig1-1} is the situation when $\delta=0.5$.    In this figure
we can observe that zero-sets around $(a_{7},a_{7})$ and $(a_{8},a_{8})$ become closer to
each other, and it seems that these two zero-sets are connected in the central 
figure, when $\delta=1$.
In the right figure, all the zero-sets around $(a_5,a_5)$ to $(a_8,a_8)$ seem connected. 
When $\delta$ becomes larger, more and more zero-sets seem to be connected with
each other (see Figure \ref{Fig2-1}).


\begin{figure}
\begin{center} \leavevmode
\includegraphics[width=0.32\textwidth]{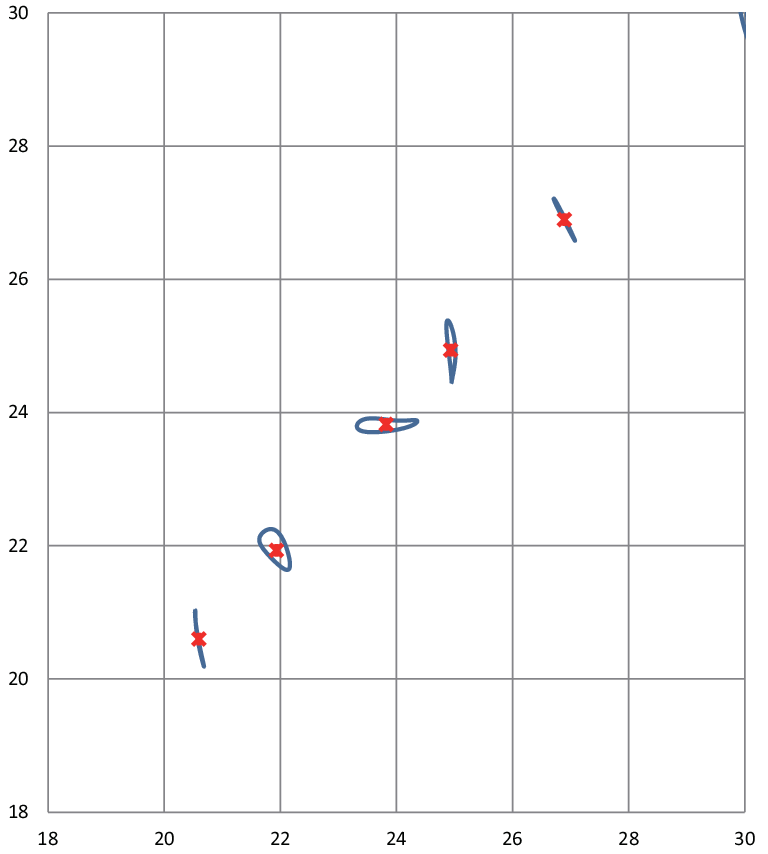}
\includegraphics[width=0.32\textwidth]{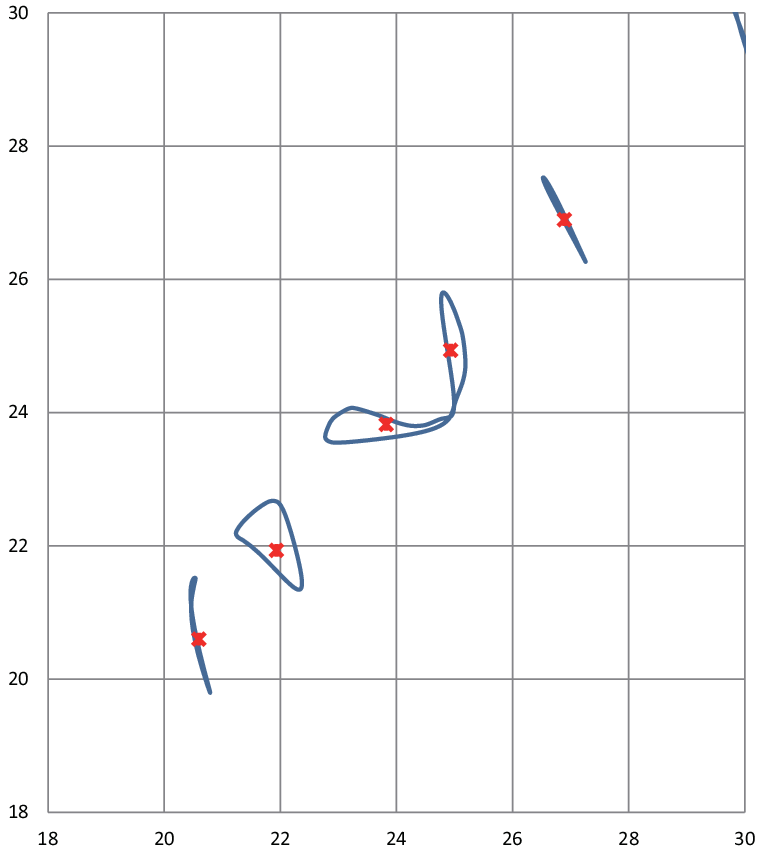}
\includegraphics[width=0.32\textwidth]{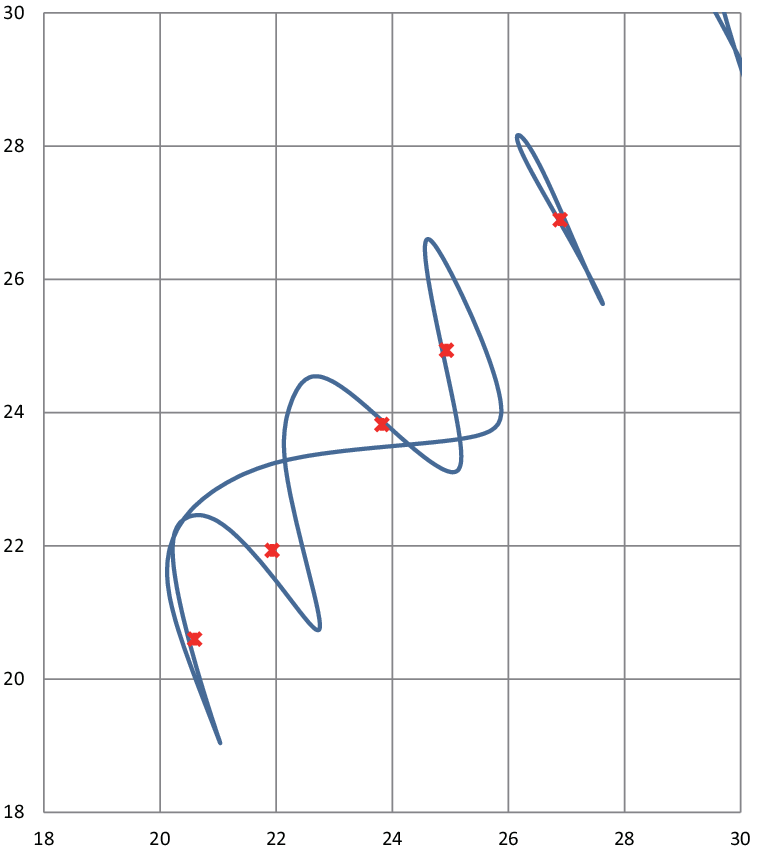}
\end{center}
\caption{The loci of the absolute values of zeros of $\zeta_2(s_1,s_2)$ 
around $(a_i,a_i)$ ($5\leq i\leq 9$) with $|s_1-s_2|=\delta$, where
$\delta=0.5$ (left), $1.0$ (center), $2.0$ (right).   The horizontal axis
represents $|s_1|$ and the vertical axis does $|s_2|$.
The points indicated by the 
cross marks are $(|a_5|,|a_5|), (|a_6|,|a_6|), \ldots, (|a_9|,|a_9|)$ from 
the lower-left to the upper-right. }
\label{Fig1-1}
\end{figure}

\begin{figure}
\begin{center} \leavevmode
\includegraphics[width=0.45\textwidth]{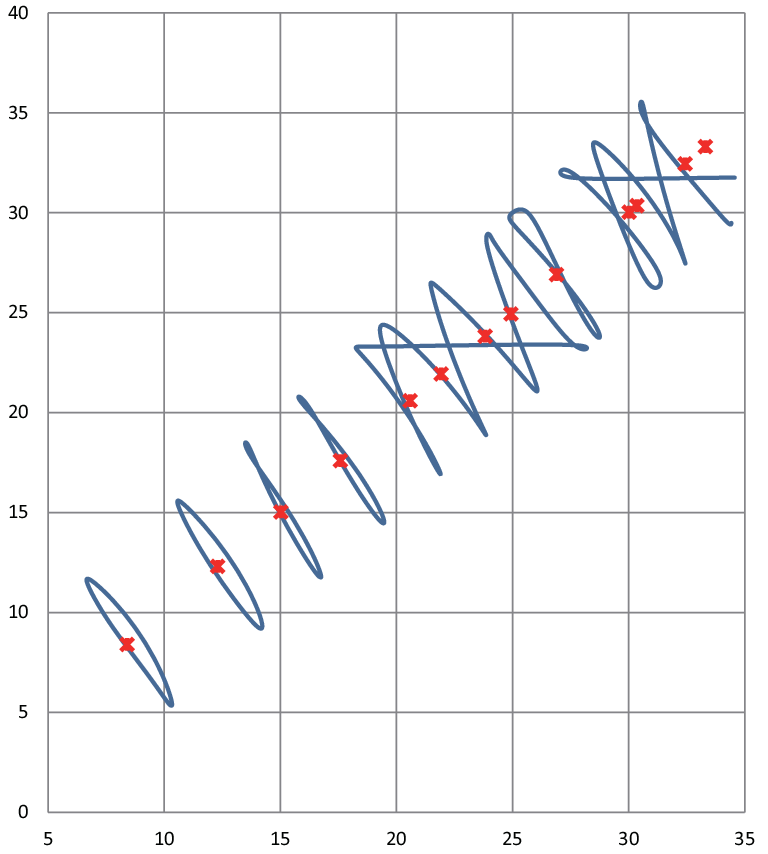}
\includegraphics[width=0.45\textwidth]{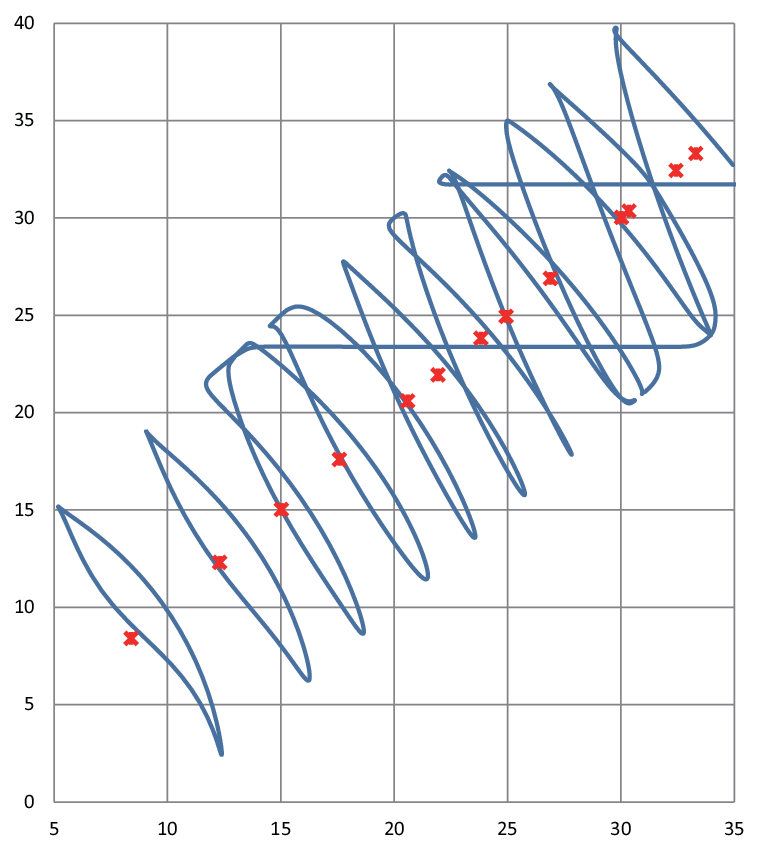}
\end{center}
\caption{The loci of the absolute values of zeros of $\zeta_2(s_1,s_2)$ 
around $(a_i,a_i)$ ($1\leq i\leq 13$) with $|s_1-s_2|=\delta$, where
$\delta=5.0$ (left), and $10.0$ (right).}
\label{Fig2-1}
\end{figure}

\begin{figure}
\begin{center} \leavevmode
\includegraphics[width=0.45\textwidth]{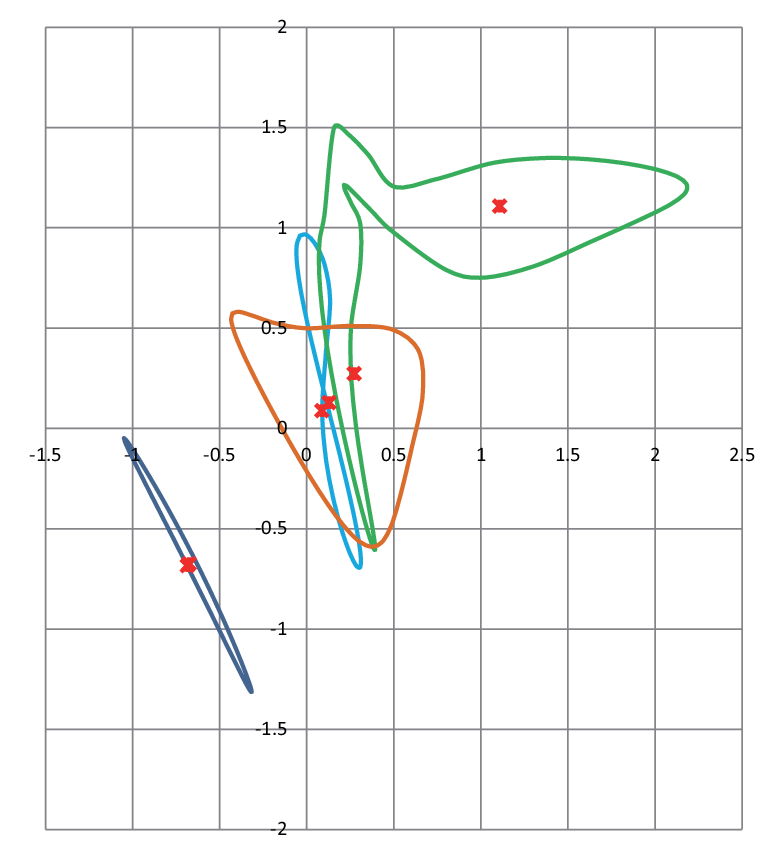}
\includegraphics[width=0.45\textwidth]{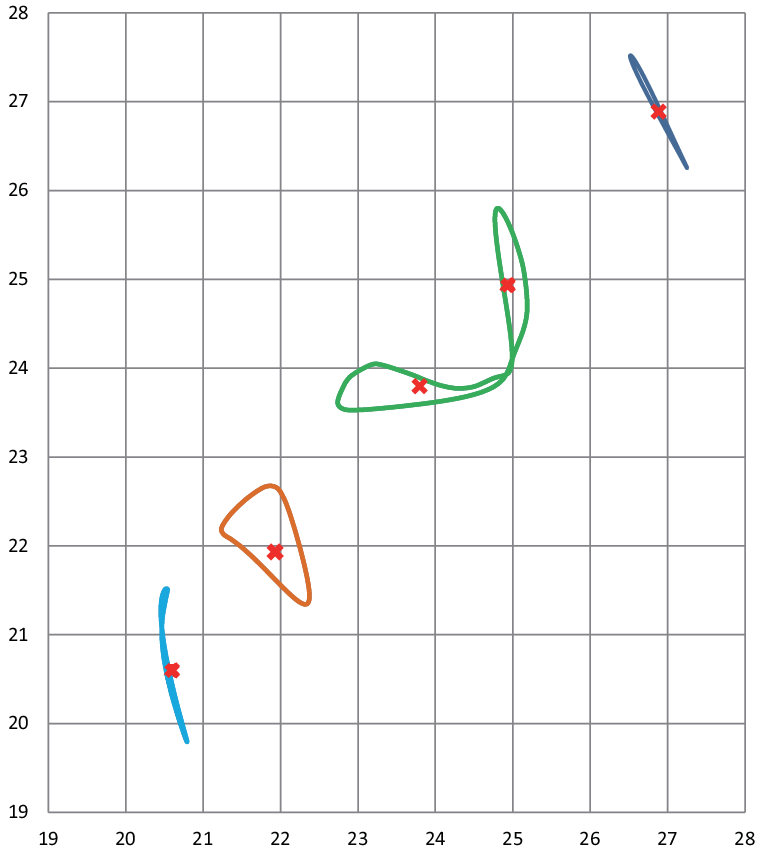}
\end{center}
\caption{The behavior of real parts and imaginary parts of zeros around 
$(a_i,a_i)$ ($5\leq i\leq 9$) when $\delta=1.0$.   On the left figure, the horizontal axis
represents $\Re s_1$, and the vertical axis does $\Re s_2$, while on the right figure 
they represent $\Im s_1$ and $\Im s_2$.   Since the absolute value is almost determined 
by the imaginary part (because the real part is relatively small), the right figure is
 very similar to the central one of Figure 1. }
\label{Fig2-2}
\end{figure}


Since Figure \ref{Fig1-1} only represents the behavior of absolute values, in order to 
make sure 
that the above zero-sets are indeed connected, it is necessary to investigate
the values of real parts and imaginary parts of those.     Figure \ref{Fig2-2} gives 
such data.
This figure shows that the loci of zeros around $(a_7,a_7)$ and 
$(a_8,a_8)$ are indeed connected.

From Figure \ref{Fig2-1} it seems that all zero-sets appearing in this figure are
connected.   (The zero-sets including the points $(a_1,a_1)$ and 
$(a_2,a_2)$ are not connected to
the other zero-sets on the figure, but further computations show that these zero-sets
look connected also, when $\delta$ becomes larger.)

From this observation, perhaps we may expect that all points
$$(a_1,a_1), (a_2,a_2), (a_3,a_3),\ldots$$ 
are lying 
on the same (unique?) zero-set.    However, later in Section 4 we will see that
the behavior of some zero-sets is rather different.    Probably it is too early 
to raise any conjecture on the global behavior of zero-sets.

\begin{rem}\label{rem-1}
In \cite[p.308]{MatSho14}, it is noted that the left-most zero in 
\cite[Figure 1]{MatSho14} is
$$(-0.830372)+i(35.603804).$$
However, there is at least one zero of $\zeta_2(s,s)$ which is located more left.    
When the authors wrote \cite{MatSho14}, they overlooked the following two zeros:
\begin{align*}
(-0.874058504) + i(44.93750365),\\
(-0.710036436) + i(53.91464901).
\end{align*}
\end{rem}

\section{Approaching the axis $s_2=0$}

From Figure \ref{Fig2-1} we can observe that, when $\delta$ becomes larger, 
the shape of the curves consisting
of zeros also becomes larger, and the bottoms of the curves look approaching to the 
horizontal axis (that is, the axis $s_2=0$).    Figure \ref{Fig3} describes 
the curves of the absolute values of zeros
when $\delta=14.0$.    In this case one curve indeed touches the horizontal axis.


\begin{figure}
\begin{center} \leavevmode
\includegraphics[width=0.5\textwidth]{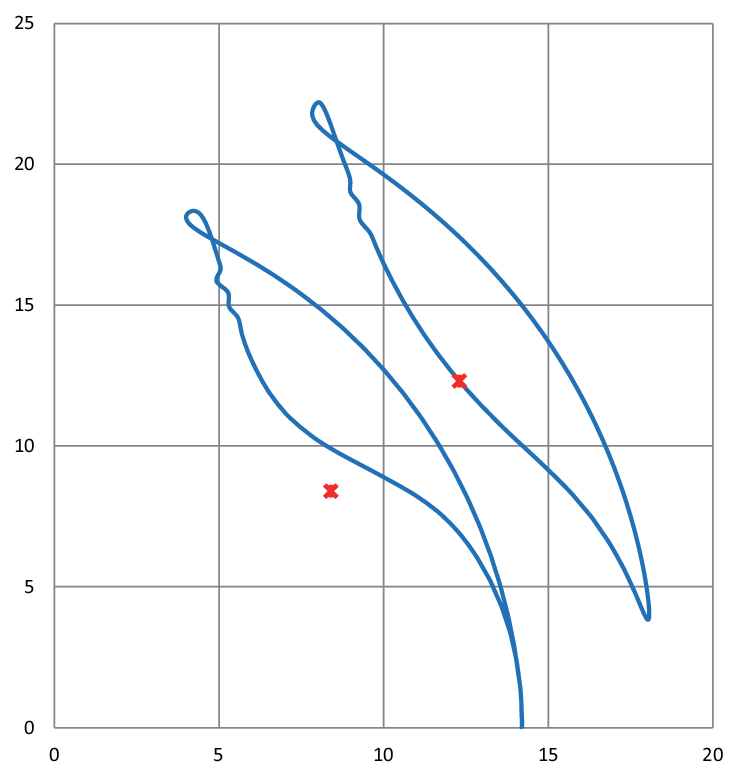}
\end{center}
\caption{The loci of the absolute values of zeros of $\zeta_2(s_1,s_2)$ 
around $(a_1,a_1)$ and $(a_2,a_2)$ with $|s_1-s_2|=\delta$, where
$\delta=14.0$.
The horizontal axis
represents $|s_1|$ and the vertical axis does $|s_2|$.}
\label{Fig3}
\end{figure}

\begin{figure}
\begin{center} \leavevmode
\includegraphics[width=0.9\textwidth]{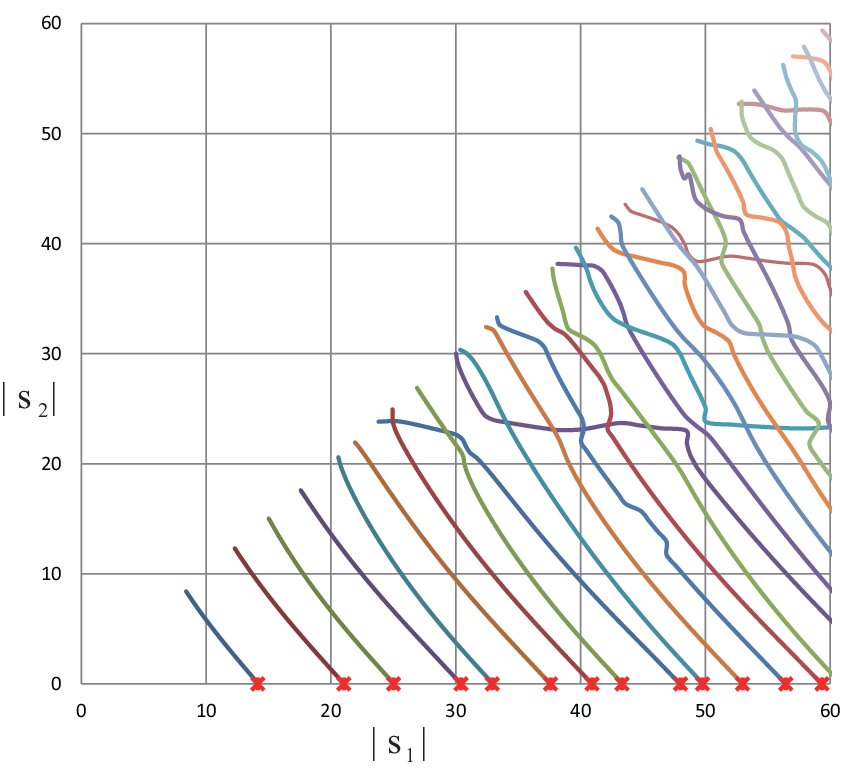}
\end{center}
\caption{The absolute values of zeros satisfying $s_2=\eta s_1$, where $1\geq\eta\geq 0$.
The horizontal axis
represents $|s_1|$ (up to $60$) and the vertical axis does $|s_2|$.
The cross marks represent the absolute values of zeros of $\zeta(s)$.}
\label{Fig4}
\end{figure}


How about the behavior of other curves?     
To investigate this point, now we use an alternative
way of calculating zeros.    Consider the equation $s_2=\eta s_1$.    In \cite{MatSho14},
we studied the zeros under the condition $\eta=1$.     Starting with the data of those 
zeros, we search for the zeros for various values of $\eta$, $1\geq \eta\geq 0$.
When $\eta\to 0$, we find that almost all curves consisting of zeros tend to the
horizontal axis (see Figure \ref{Fig4}).


\begin{table}
\caption{The left is the list of the values of $s_1$ of the points 
where the curves touch the horizontal axis.    The right is the list of 
non-trivial zeros of $\zeta(s)$.}
\label{table-1}
\begin{center}
\begin{tabbing}
\hspace{0.3cm}\= \hspace{8cm} \= \kill
\> $ (1.247595281) + i (14.14857043) $ \> $ (0.5) + i (14.13472514) $ \\
\> $ (1.279113136) + i (21.01244258) $ \> $ (0.5) + i (21.02203964) $ \\
\> $ (1.292752716) + i (25.03054326) $ \> $ (0.5) + i (25.01085758) $ \\
\> $ (1.304538379) + i (30.39099998) $ \> $ (0.5) + i (30.42487613) $ \\
\> $ (1.310484619) + i (32.97276761) $ \> $ (0.5) + i (32.93506159) $ \\
\> $ (1.319356152) + i (37.57182573) $ \> $ (0.5) + i (37.58617816) $ \\
\> $ (1.320370441) + i (40.90318345) $ \> $ (0.5) + i (40.91871901) $ \\
\> $ (1.328328378) + i (43.36573059) $ \> $ (0.5) + i (43.32707328) $ \\
\> $ (1.333209526) + i (47.94829355) $ \> $ (0.5) + i (48.00515088) $ \\
\> $ (1.330151768) + i (49.81347595) $ \> $ (0.5) + i (49.77383248) $ \\
\> $ (1.338852528) + i (52.98250152) $ \> $ (0.5) + i (52.97032148) $ \\
\> $ (1.341784096) + i (56.42981942) $ \> $ (0.5) + i (56.44624769) $ \\
\> $ (1.336083544) + i (59.30477438) $ \> $ (0.5) + i (59.34704400) $ \\
\> $ (1.347599692) + i (60.90235918) $ \> $ (0.5) + i (60.83177852) $ \\
\> $ (1.350310992) + i (65.06828349) $ \> $ (0.5) + i (65.11254405) $ \\
\> $ (1.343496503) + i (67.09040611) $ \> $ (0.5) + i (67.07981053) $ \\
\> $ (1.348806781) + i (69.55741541) $ \> $ (0.5) + i (69.54640171) $ \\
\> $ (1.354734356) + i (72.08670993) $ \> $ (0.5) + i (72.06715767) $ \\
\> $ (1.356705834) + i (75.63889778) $ \> $ (0.5) + i (75.70469069) $ \\
\> $ (1.344904017) + i (77.17673184) $ \> $ (0.5) + i (77.14484007) $ \\
\> $ (1.360065509) + i (79.37508073) $ \> $ (0.5) + i (79.33737502) $ \\
\> $ (1.360772059) + i (82.86977267) $ \> $ (0.5) + i (82.91038085) $ \\
\> $ (1.355304386) + i (84.75005063) $ \> $ (0.5) + i (84.73549298) $ \\
\> $ (1.353373038) + i (87.38999143) $ \> $ (0.5) + i (87.42527461) $ \\
\> $ (1.365347698) + i (88.87469754) $ \> $ (0.5) + i (88.80911121) $ \\
\> $ (1.367563235) + i (92.45429247) $ \> $ (0.5) + i (92.49189927) $ \\
\> $ (1.352878941) + i (94.60615786) $ \> $ (0.5) + i (94.65134404) $ \\
\> $ (1.364213735) + i (95.93955674) $ \> $ (0.5) + i (95.87063423) $ \\
\> $ (1.367791276) + i (98.82956582) $ \> $ (0.5) + i (98.83119422) $ \\
\end{tabbing}
\end{center}
\end{table}


From Figure \ref{Fig4} we observe that the values $|s_1|$ of the points $(s_1,s_2)$
at which the loci touch the
horizontal axis are almost the same as the absolute values of zeros of $\zeta(s)$.
Let us list up those values.
The left of Table \ref{table-1} is the list of the values of $s_1$ of the points 
where the curves touch the horizontal axis.    Comparing this table with the list of  
non-trivial zeros of $\zeta(s)$ (the right of Table \ref{table-1}), we find:

\begin{obs}\label{obs-1}
The imaginary part of each $s_1$ in the left list of Table \ref{table-1} is very close 
to an imaginary part of a value appearing in the right list of Table \ref{table-1},
that is, a non-trivial zero of $\zeta(s)$. 
\end{obs}

The following argument is not rigorous, but at least heuristically, explains
this observation.
 
Recall the formula \cite[(2.3)]{MatSho14} (originally in \cite{AET}): 
\begin{align}\label{3-1}
&\zeta_2(s_1,s_2)=\frac{\zeta(s_1+s_2-1)}{s_2-1}-\frac{\zeta(s_1+s_2)}{2}\\
&\qquad+\sum_{q=1}^l (s_2)_q\frac{B_{q+1}}{(q+1)!}\zeta(s_1+s_2+q)
-\sum_{n_1=1}^{\infty}\frac{\phi_l(n_1,s_2)}{n_1^{s_1}},\notag
\end{align}
where $(s_2)_q=s_2(s_2+1)\cdots(s_2+q-1)$, $B_{q+1}$ is the $(q+1)$-th Bernoulli number,
and 
\begin{align*}
\lefteqn{\phi_l(n_1,s_2)}\\
&=\sum_{k=1}^n\frac{1}{k^s}-\left\{\frac{n^{1-s}-1}{1-s}
+\frac{1}{2n^s}-\sum_{q=1}^l\frac{(s)_q B_{q+1}}{(q+1)! n^{s+q}}+\zeta(s)
-\frac{1}{s-1}\right\}
\end{align*}
(see \cite[(2.4)]{MatSho14}).
Put $s_2=0$ in \eqref{3-1}.    Since $\phi_l(n_1,0)=0$ (which can be seen by
\cite[(2.5)]{MatSho14}), we find that the two sums on the right-hand side of
\eqref{3-1} are both zero, so
\begin{align}\label{3-2}
\zeta_2(s_1,0)=-\zeta(s_1-1)-\frac{\zeta(s_1)}{2}.
\end{align}  
Now, let $s_1^*$ be one of the values in the left list of Table 1.
Then $\zeta_2(s_1^*,0)=0$, so \eqref{3-2} implies
\begin{align}\label{3-3}
\zeta(s_1^*)=-2\zeta(s_1^*-1).
\end{align}

Let
$$
C(t)=\{\zeta(\sigma+it)\;|\;\sigma\in\mathbb{R}\}.
$$
If $t$ is very close to the imaginary part of a non-trivial zero of $\zeta(s)$, the graph
of the curve $C(t)$ passes very close to the origin when $\sigma=1/2$ 
(see Figure \ref{Fig5}).
As can be seen from Figure \ref{Fig5}, when $\sigma$ moves, the slope of the graph 
of $C(t)$ does not
so rapidly change.    This can be naturally expected, because the approximate
functional equation of $\zeta(s)$ (see \cite[Theorem 4.15]{Tit51}) implies that
the behavior of $\zeta(s)$ is dominated by the terms of the form 
$n^{-s}=n^{-\sigma}e^{-it \log n}$, and if $t$ is fixed, then the ``argument" part of
these terms does not change.


\begin{figure}
\begin{center} \leavevmode
\includegraphics[width=0.9\textwidth]{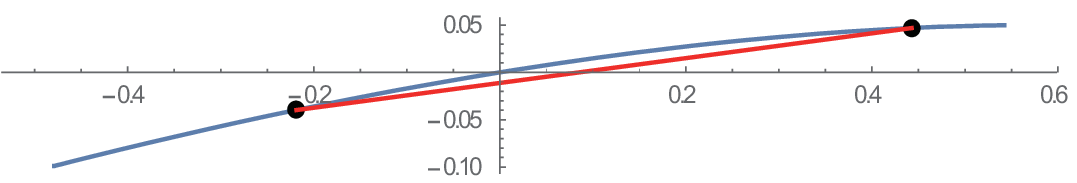} \par \bigskip
\includegraphics[width=0.9\textwidth]{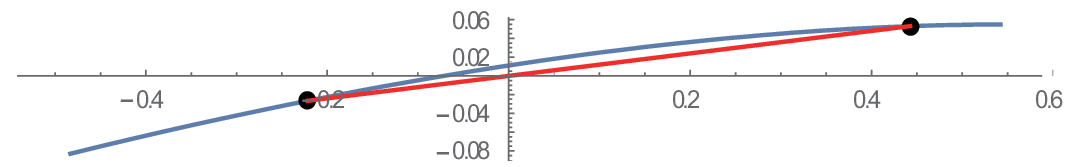}
\end{center}
\caption{The first one is the locus of $\zeta(\sigma+it_0)$, where 
$t_0=14.13472514$ (= the first of the right list of Table \ref{table-1}).   Therefore this 
locus crosses the origin.    The two black-dots represent the points 
$s(t_0)=(0.442629) + i(0.0469269)$ and
$s(t_0-1)=(-0.218684) -i(0.0397677)$, with
$\sigma(t_0)=1.2475$.   
This satisfies \eqref{3-4}.
The second one is the locus of $\zeta(\sigma+it_0^*)$,
where $t_0^*=14.14857043$ (= the first of the left list of Table \ref{table-1}),
for which $L(t_0^*)$ crosses the origin.}
\label{Fig5}
\end{figure}


Therefore we can find a number $\sigma(t)\in\mathbb{R}$ such that
\begin{align}\label{3-4}
|\zeta(s(t))|=2|\zeta(s(t)-1)|,
\end{align} 
where $s(t)=\sigma(t)+it$.    We denote by $L(t)$ the segment joining $\zeta(s(t))$ and
$\zeta(s(t)-1)$. 

Now choose $t=t_0$, for which
$(1/2)+it_0$ is in the right list of Table \ref{table-1}.    Then, as in the first one of 
Figure \ref{Fig5} (where the case $t_0=14.13472514$ is described), 
$\arg\zeta(s(t_0)-1)$ is 
almost equal to $\arg\zeta(s(t_0))\pm\pi$, so $L(t_0)$ passes very close to the origin.
(If $C(t_0)$ would be a straight line, then $L(t_0)$ could indeed cross the origin; but 
this is not the case.)
Move the value of $t$ a little from $t_0$.    Then the curve $C(t)$ also moves a little.
Then, as in the second one of Figure \ref{Fig5}, we may find a value of $t=t_0^*$, 
close to $t_0$, for which
$L(t_0^*)$ indeed crosses the origin.    This implies    
$\zeta(s(t_0^*))=-2\zeta(s(t_0^*)-1)$, that is, in view of \eqref{3-3}, this 
$s(t_0^*)=\sigma(t_0^*)+it_0^*$ 
should be in the left list of Table \ref{table-1}.

It is an interesting problem to make the above argument more rigorous, and to get some
formula which expresses the phenomenon of Observation \ref{obs-1}.

\begin{rem}
It is also observed that the real parts of the points on the list of Table 1 are
close to each other, around the value $1.3$.
So far we have not found any theoretical reasoning of this 
phenomenon.
\end{rem}

\section{Approaching the zeros of $\zeta(s)=1$}

In Figure \ref{Fig4}, we can observe that almost all curves approach the horizontal axis.
However, when we extend the range of computations, we 
find that there are curves which do not seem to approach the
horizontal axis (Figure \ref{Fig4-2}).    Along these curves, it seems that 
$|s_1|$ becomes larger and larger.


\begin{figure}
\begin{center} \leavevmode
\includegraphics[width=0.9\textwidth]{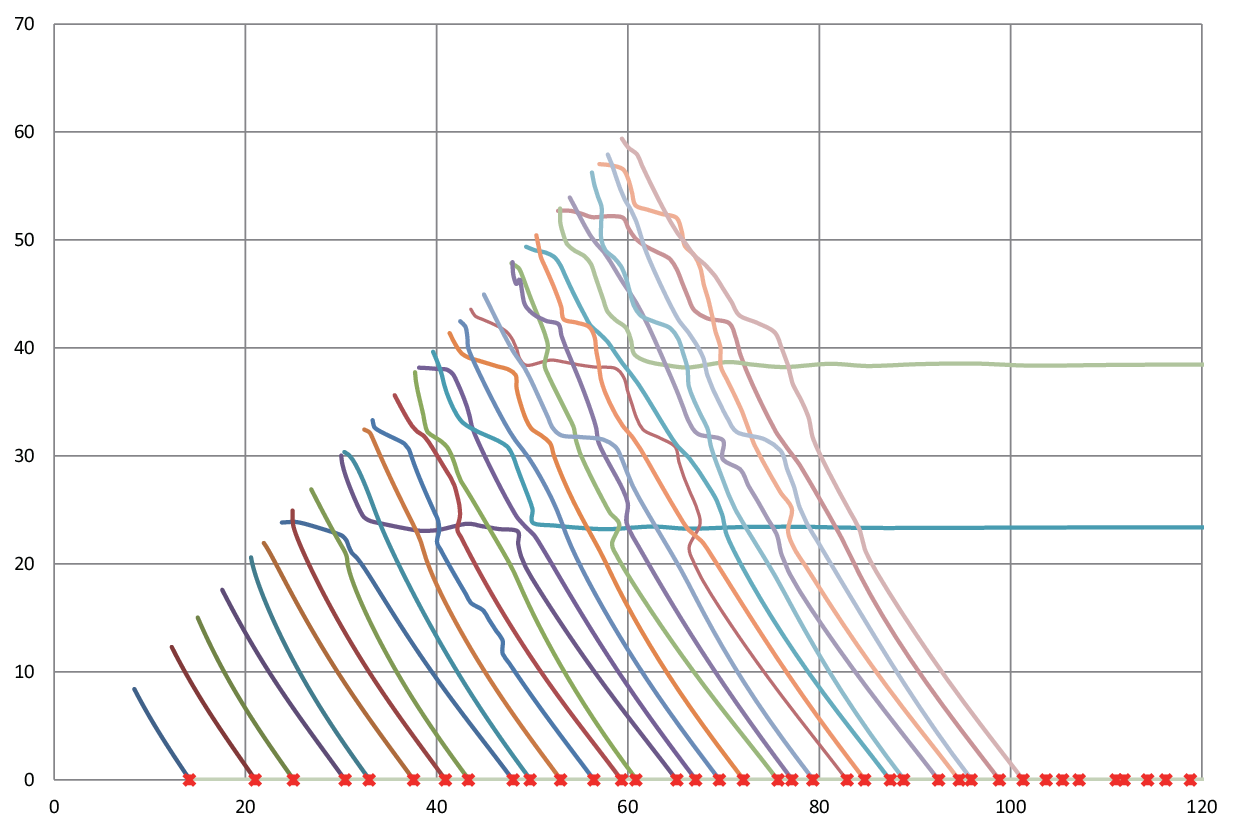}
\end{center}
\caption{This figure represents the same computations as in Figure \ref{Fig4}, but the 
range of $|s_1|$ is up to $120$.   
The cross marks represent the absolute values of zeros of $\zeta(s)$.
There appear two curves which do not approach the 
horizontal axis.}
\label{Fig4-2}
\end{figure}


Moreover, extending the range of computations, we can find more zero-sets, along them
$|s_1|$ seems to tend to infinity (see Figure \ref{Fig7-1}).


\begin{figure}
\begin{center} \leavevmode
\includegraphics[width=0.6\textwidth]{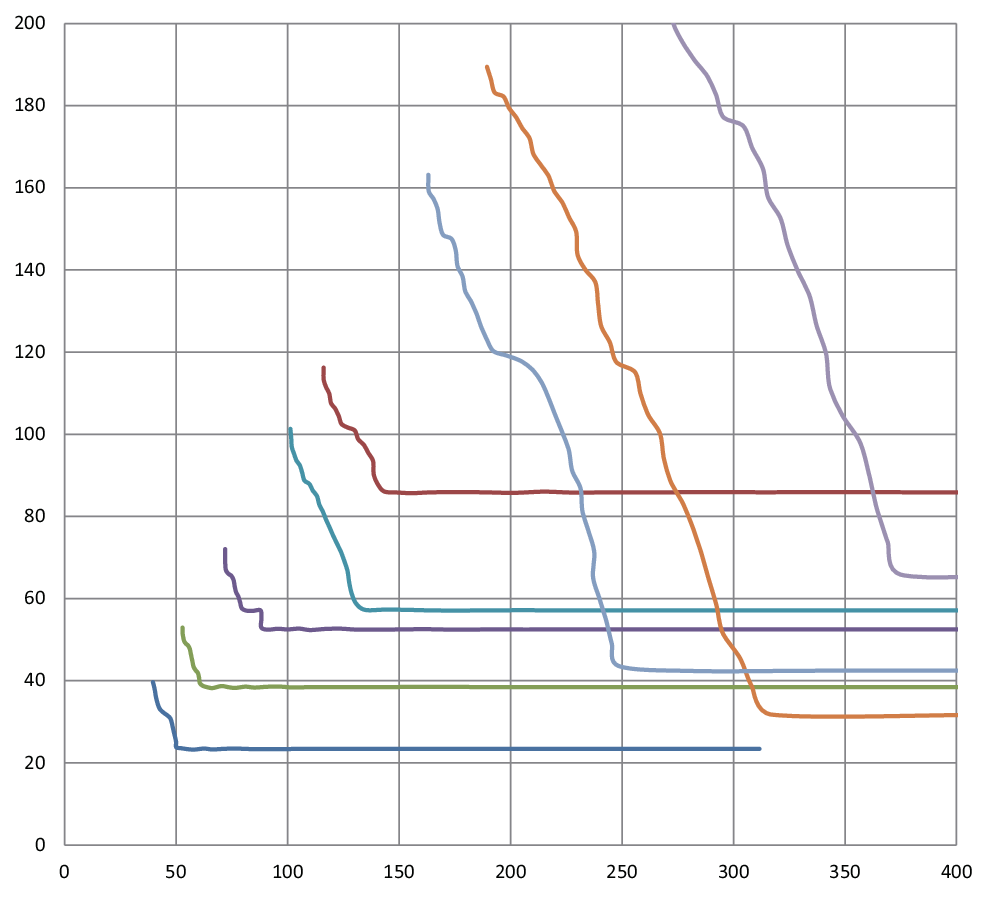}
\end{center}
\caption{The same computations as in Figure \ref{Fig4}, but the range of $|s_1|$ is up to 
$400$.   Only the curves which do not approach the horizontal axis are drawn.}
\label{Fig7-1}
\end{figure}


The numerical data suggests that $\Re s_1$ tends to infinity
along these curves (the left of Figure \ref{Fig7-2}), while
$\Re s_2$ remains finite (the right of Figure \ref{Fig7-2}). 

\begin{figure}
\begin{center} \leavevmode
\includegraphics[width=0.45\textwidth]{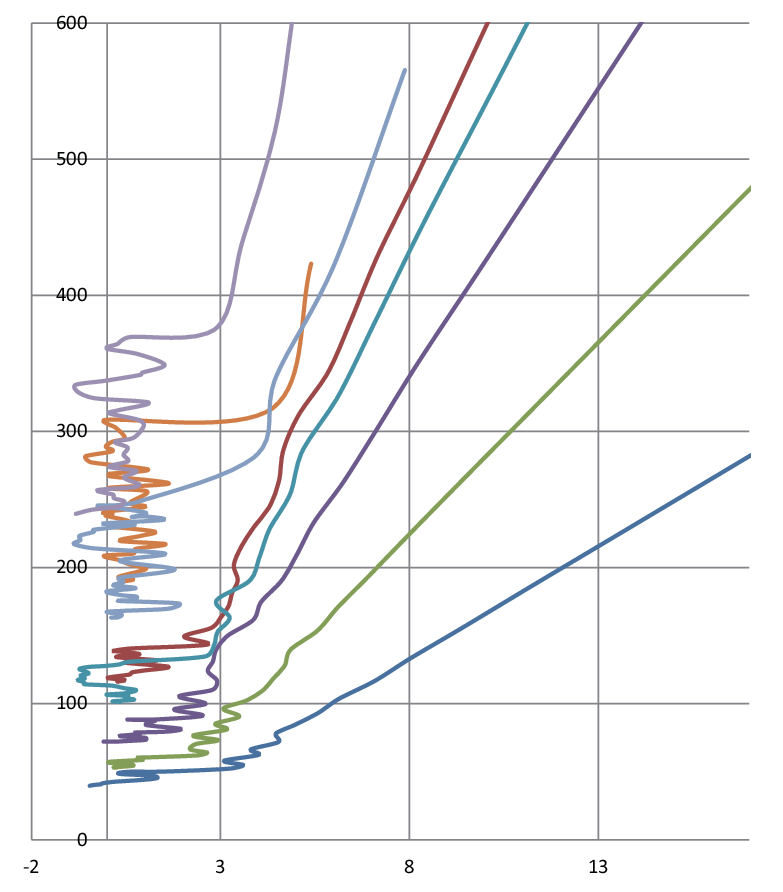}
\includegraphics[width=0.45\textwidth]{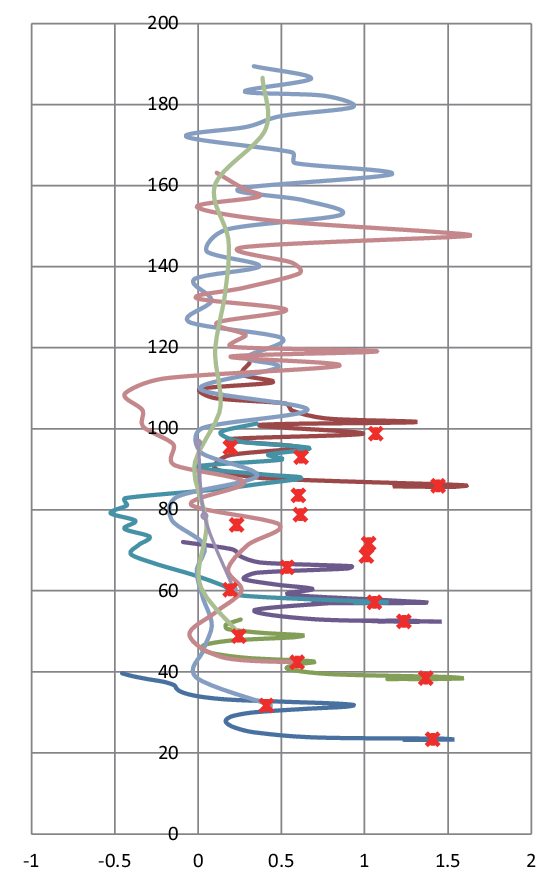}
\end{center}
\caption{The behavior of $s_1$ (left) and $s_2$ (right) of curves described in
Figure \ref{Fig7-1}.   The cross marks on the right figure are solutions of 
$\zeta(s_2)=1$.    There may be
more solutions, but we only mark the solutions we found numerically. }
\label{Fig7-2}
\end{figure}


What happens?    We can prove the following facts.

\begin{prop}\label{prop1}
{\rm (i)}
{\it Let $(s_1^{(m)},s_2^{(m)})$ $(m=1,2,\ldots)$ be a sequence of points on a zero-divisor 
of $\zeta_2(s_1,s_2)$.   If $\sigma_1^{(m)}=\Re s_1^{(m)}$ tends to infinity and 
$|s_2^{(m)}|$ remains bounded {\rm(}and not close to $1${\rm)} as
$m\to\infty$, then $s_2^{(m)}$ tends to a solution of $\zeta(s_2)=1$}.

{\rm (ii)}
{\it Conversely, for any solution $\rho_2$ of $\zeta(\rho_2)=1$ and any $\varepsilon>0$, 
we can find a zero of $\zeta_2(s_1,s_2)$ such that $|s_2-\rho_2|<\varepsilon$}.
\end{prop}

This result is due to Professor Seidai Yasuda (Osaka University). 
The authors express their sincere gratitude to him for the permission of including his
result in the present paper.

In the next section we will state general theorems on the asymptotic behavior of
zero-sets of $r$-fold sum \eqref{1-2}.    The above proposition is just a special case of
those theorems.

\begin{rem}\label{rem_trivial}
The assertion (i) of the above proposition is, if $\Re s_2>1$, obvious.
Because in this case we can use \eqref{1-1}.    Letting $\Re s_1\to\infty$ on \eqref{1-1},
we see that the right-hand side tends to 
$\sum_{n_2=1}^{\infty}(1+n_2)^{-s_2}=\zeta(s_2)-1$.
\end{rem}

\section{The asymptotic behavior of zero-sets of the $r$-fold sum}

We study the behavior of $\zeta_r(s_1,\ldots,s_r)$ when $\sigma_k=\Re s_k$ for some $k$
tends to $+\infty$ 
while the other variables remain bounded.    
We write $\mathbb{C}^r=\mathbb{C}_{(k)}^{r-1}\times\mathbb{C}_{(k)}$, 
where $\mathbb{C}_{(k)}^{r-1}$ consists of all $(s_1,\ldots,s_{k-1},s_{k+1},\ldots,s_r)$
and $\mathbb{C}_{(k)}$ denotes the $s_k$-plane.
Let $S_{k,r}$ be the union of all hyperplanes $A_{j,n}$ and $A_1$ (defined in 
Section \ref{sec-1}) whose defining equation does not include $s_k$.   That is,
$$
S_{k,r}=\Bigl(\bigcup_{2\leq j\leq r-k}\bigcup_{n=0}^{\infty}A_r^{(j,n)}\Bigr)
\bigcup A_r^{(1)} \qquad (1\leq k\leq r-2),
$$
$S_{r-1,r}=A_r^{(1)}$, and $S_{r,r}=\emptyset$.

First, when $k\geq 2$, the conclusion
is simple.

\begin{thm}\label{th-y-0}
Let $r\geq 2$, $2\leq k\leq r$, and $E_{k,r-1}$ a closed subset of 
$\mathbb{C}_{(k)}^{r-1}$
such that
$$
\{(\sigma_1,\ldots,\sigma_{k-1},\sigma_{k+1},\ldots,\sigma_r)\;|\;
(s_1,\ldots,s_{k-1},s_{k+1},\ldots,s_r)\in E_{k,r-1}\}
$$
is bounded, and
\begin{align}\label{sing_condition}
(E_{k,r-1}\times \mathbb{C}_{(k)})\cap S_{k,r}=\emptyset.
\end{align}
When $\sigma_k=\Re s_k$ tends to $+\infty$ while the other variables remain
in $E_{k,r-1}$, 
the value
$\zeta_r(s_1,\ldots,s_r)$ tends to $0$, uniformly in $E_{k,r-1}$.
\end{thm}


More interesting is the situation when $k=1$.
Let $r\geq 1$, and define
\begin{align}\label{y-0}
F_{r-1}(s_2,\ldots,s_r)=\sum_{j=1}^{r-1}(-1)^{r-j+1}\zeta_j(s_{r-j+1},\ldots,s_r)-(-1)^r.
\end{align}
{\rm (}When $r=1$, we understand simply that $F_0=1$.{\rm)}
Denote by $\mathcal{H}_{r-1}$ the hypersurface in the space $\mathbb{C}_{(1)}^{r-1}$ 
defined by 
the equation 
\begin{align}\label{y-1}
F_{r-1}(s_2,\ldots,s_r)=0.
\end{align}

\begin{thm}\label{th-y-1}
Let $D_{r-1}$ a bounded closed subset of 
$\mathbb{C}_{(1)}^{r-1}$
such that 
\begin{align}\label{sing_condition2}
(D_{r-1}\times \mathbb{C}_{(1)})\cap S_{1,r}=\emptyset.
\end{align}
Let $(s_1^{(m)},\ldots,s_r^{(m)})$ $(m=1,2,3,\ldots)$ be a sequence of points in $\mathbb{C}^r$.
Assume that, as $m\to\infty$, $\sigma_1^{(m)}=\Re s_1^{(m)}$ tends to $+\infty$ while 
$(s_2^{(m)},\ldots,s_r^{(m)})$ remains in $D_{r-1}$.
Then for any $\varepsilon>0$, we can choose a sufficiently large $M$, uniformly
in $D_{r-1}$, for which
\begin{align}\label{y-1-1}
|\zeta_r(s_1^{(m)},\ldots,s_r^{(m)})-F_{r-1}(s_2^{(m)},\ldots,s_r^{(m)})|<\varepsilon
\end{align}
holds for any $m\geq M$.
\end{thm}

\begin{thm}\label{th-y-2}  Assume $r\geq 2$.

{\rm (i)} Let $(s_1^{(m)},\ldots,s_r^{(m)})$ $(m=1,2,3,\ldots)$ be a sequence of points on a
zero-set of $\zeta_r(s_1,\ldots,s_r)$.    Assume that $\sigma_1^{(m)}=\Re s_1^{(m)}$ tends to
$+\infty$ while
$(s_2^{(m)},\ldots,s_r^{(m)})$ remains in $D_{r-1}$ {\rm(}as in the statement of Theorem \ref{th-y-1}{\rm)}.
Then, uniformly on $D_{r-1}$, the sequence $(s_2^{(m)},\ldots,s_r^{(m)})$ is approaching 
the hypersurface $\mathcal{H}_{r-1}$.
Therefore, we can find a subsequence $(s_2^{(k_m)},\ldots,s_r^{(k_m)})$ $(m=1,2,3,\ldots)$
which converges to a point $(\rho_2,\ldots,\rho_r)$ on $\mathcal{H}_{r-1}$.

{\rm (ii)} Conversely, for any solution $(\rho_2,\ldots,\rho_r)$ of \eqref{y-1}, we can find
a sequence $(s_1^{(m)},\ldots,s_r^{(m)})$ $(m=1,2,3,\ldots)$ on a 
zero-set of $\zeta_r(s_1,\ldots,s_r)$ which converges to $(\rho_2,\ldots,\rho_r)$.
\end{thm}

The case $r=2$ of Theorem \ref{th-y-2} is exactly Proposition \ref{prop1} given 
in the previous section.

\section{Proofs of theorems}

We begin with the proof of Theorem \ref{th-y-0}.    

\begin{proof}[Proof of Theorem \ref{th-y-0}]
We divide the proof into several steps.

{\it Step 1}.$\;$
The multiple series
\begin{align}\label{y-2}
\zeta_r(s_1,\ldots,s_r)=\sum_{n_1=1}^{\infty}\cdots\sum_{n_r=1}^{\infty}n_1^{-s_1}
(n_1+n_2)^{-s_2}\cdots(n_1+\cdots+n_r)^{-s_r}
\end{align}
($r\geq 2$) is absolutely convergent when $\sigma_r$ is sufficiently large.    
Since $n_1+\cdots+n_r\geq r\geq 2$, all 
the terms on the right-hand side tend to 0 when $\sigma_r\to\infty$, uniformly in 
$(s_1,\ldots,s_{r-1})\in E_{r,r-1}$. 
Hence the assertion in the case $k=r$ follows.

{\it Step 2}.$\;$
However in the case $\sigma_k\to\infty$ $(2\leq k\leq r-1)$, the
above argument is not enough,
because if $\sigma_r$ is not large, then the expression \eqref{y-2} is not valid.
Therefore we first carry out the meromorphic continuation.    

We prove the continuation by induction.    Assume that
$\zeta_{r-1}(s_1,\ldots,s_{r-1})$ can be continued meromorphically to the whole
space $\mathbb{C}^{r-1}$.

Here, we apply the method of
using the Mellin-Barnes integral formula
\begin{align}\label{y-3}
(1+\lambda)^{-s}=\frac{1}{2\pi i}\int_{(c)}\frac{\Gamma(s+z)\Gamma(-z)}{\Gamma(s)}
\lambda^z dz,
\end{align}
where $s,\lambda\in\mathbb{C}$, $\Re s>0$, $\lambda\neq 0$, $|\arg\lambda|<\pi$, 
$-\Re s<c<0$, and the path of
integration is the vertical line $\Re z=c$. 
The following argument was done (in a more general form) in \cite{Mat03} \cite{MatJNT}.  
Assume, at first, that $\Re s_k>1$ for all $1\leq k\leq r$.
Then \eqref{y-2} is absolutely convergent, and an application of \eqref{y-3} yields
\begin{align}\label{y-4}
&\zeta_r(s_1,\ldots,s_r)=\frac{1}{2\pi i}\int_{(c)}\frac{\Gamma(s_r+z)\Gamma(-z)}{\Gamma(s_r)}\\
&\;\times\zeta_{r-1}(s_1,\ldots,s_{r-2},s_{r-1}+s_r+z)\zeta(-z)dz \notag
\end{align}
(see \cite[(12.3)]{Mat03}), where $-\Re s_r<c<-1$.   Then we shift the path of integration to 
the line $\Re z=M-\varepsilon$, where $M$ is a large positive integer and $\varepsilon$ is a
small positive number.    Counting the relevant residues, we obtain
\begin{align}\label{y-5}
&\zeta_r(s_1,\ldots,s_r)=\frac{1}{s_r-1}\zeta_{r-1}(s_1,\ldots,s_{r-2},s_{r-1}+s_r-1)\\
&+\sum_{l=0}^{M-1}\binom{-s_n}{l}\zeta_{r-1}(s_1,\ldots,s_{r-2},s_{r-1}+s_r+l)\zeta(-l)\notag\\
&+\frac{1}{2\pi i}\int_{(M-\varepsilon)}\frac{\Gamma(s_r+z)\Gamma(-z)}{\Gamma(s_r)}\notag\\
&\quad\times\zeta_{r-1}(s_1,\ldots,s_{r-2},s_{r-1}+s_r+z)\zeta(-z)dz \notag
\end{align}
(see \cite[(12.7)]{Mat03}).
Since the last integral is convergent in the wider region
\begin{align}\label{y-6}
\{(s_1,\ldots,s_r)\;|\;\Re(s_{r-j+1}+\cdots+s_r)>j-1-M+\varepsilon\; (1\leq j\leq r)\}
\end{align}
(see \cite[(12.9)]{Mat03}), we find that $\zeta_r(s_1,\ldots,s_r)$ can be continued to
this region.
Since $M$ is arbitrary, \eqref{y-5} implies the meromorphic
continuation of $\zeta_r(s_1,\ldots,s_r)$ to the whole space $\mathbb{C}^r$.

{\it Step 3}.$\;$
Let $2\leq k\leq r-1$,
and we prove the theorem in this case by induction on $r$.   
Note that the theorem in the case $r=2$ has been already proved in Step 1.
Assume the theorem is true for $r-1$,
and consider the situation:

$(*)$  
$\sigma_k=\Re s_k$ tends to $+\infty$, while the other variables remain
in the region $E_{k,r-1}$.  
  
Assume that $(s_1,\ldots,s_r)$ is under the situation $(*)$.   Then it is clearly 
included in the above region \eqref{y-6}
for sufficiently large $M$, so we can use the expression \eqref{y-5}.
Therefore our aim is to show that the right-hand side of \eqref{y-5} tends to 0
when $\sigma_k\to\infty$.

First note that the factor $(s_r-1)^{-1}$ remains bounded.    This is because
\eqref{sing_condition} especially implies
$(E_{k,r-1}\times\mathbb{C}_{(k)})\cap A_r^{(1)}=\emptyset$.

When $k=r-1$, we see that the real part of the last variable tends to $+\infty$ in all
the $\zeta_{r-1}(\cdot)$ factors on the right-hand side of \eqref{y-5}.
Therefore we can conclude that the right-hand side of \eqref{y-5} tends to 0 as
$\sigma_{r-1}\to\infty$, as in Step 1.

{\it Step 4}.$\;$
Assume $2\leq k\leq r-2$.   Then we can define
\begin{align*}
E_{k,r-1}^w=\left\{(s_1,\ldots,s_{k-1},s_{k+1},\ldots,s_{r-2},s_{r-1}+s_r+w)\;|\right.\\
\left.\quad(s_1,\ldots,s_{k-1},s_{k+1},\ldots,s_r)\in E_{k,r-1}\right\}\subset\mathbb{C}^{r-2},
\end{align*}
where $w=-1,0,1,2,\ldots,M-1$.
These $E_{k,r-1}^w$ are all closed subsets whose real parts are bounded.

We claim 
\begin{align}\label{claim}
(E_{k,r-1}^w\times\mathbb{C}_{(k)})\cap S_{k,r-1}=\emptyset
\qquad(w=-1,0,1,2,\ldots,M-1).
\end{align}

In fact, if 
\begin{align*}
\lefteqn{(s_1,\ldots,s_{r-2},s_{r-1}+s_r+w)}\\
&=((s_1,\ldots,s_{k-1},s_{k+1},\ldots,s_{r-2},s_{r-1}+s_r+w),s_k)\\
&\in (E_{k,r-1}^w\times\mathbb{C}_{(k)})\cap S_{k,r-1},
\end{align*}
then $(s_1,\ldots,s_{r-2},s_{r-1}+s_r+w)\in A_{r-1}^{(j,n)}$ for some $j\in\{2,\ldots,r-1-k\}$
and some $n\geq 0$, or $(s_1,\ldots,s_{r-2},s_{r-1}+s_r+w)\in A_{r-1}^{(1)}$.
The former case implies
$$
s_{r-1-j+1}+\cdots+s_{r-2}+(s_{r-1}+s_r+w)=j-n,
$$
so $s_{r-j}+\cdots+s_r=j-n-w$, and the latter case implies
$s_{r-1}+s_r=1-w$.
Therefore we conclude that $(s_1,\ldots,s_r)\in S_{k,r}$, which contradicts with the 
assumption \eqref{sing_condition}.    The claim \eqref{claim} follows.

Therefore we can apply the induction assumption to the factors
$$
\zeta_{r-1}(s_1,\ldots,s_{r-2},s_{r-1}+s_r+w)\qquad (w=-1,0,1,2,\ldots,M-1)
$$
on the right-hand side of \eqref{y-5} to find that these factors tends to 0
under the assumption $(*)$.

Lastly we consider the integral term.    We may choose $M$ in \eqref{y-5} sufficiently
large so that the factor 
$\zeta_{r-1}(s_1,\ldots,s_{r-2},s_{r-1}+s_r+z)$ in the integrand is convergent
absolutely.    Then
\begin{align*}
\lefteqn{\zeta_{r-1}(s_1,\ldots,s_{r-2},s_{r-1}+s_r+z)}\\
&=\sum_{n_1=1}^{\infty}\cdots\sum_{n_{r-1}=1}^{\infty}n_1^{-s_1}\cdots
(n_1+\cdots+n_{r-2})^{-s_{r-2}}(n_1+\cdots+n_{r-1})^{-s_{r-1}-s_r-z},
\end{align*}
and since $n_1+\cdots+n_k\geq k\geq 2$, this vanishes when $\sigma_k\to\infty$.
This completes the proof of the theorem.

%
\end{proof}

Next we proceed to the proof of Theorem \ref{th-y-1}.    The assertion (i) of 
Theorem \ref{th-y-2} is clearly
a special case of Theorem \ref{th-y-1}, that is, the case
$\zeta_r(s_1^{(m)},\ldots,s_r^{(m)})=0$ in \eqref{y-1-1}.

\begin{proof}[Proof of Theorem \ref{th-y-1}]    
When $r=1$, the assertion of the theorem is $\zeta(s_1^{(m)})\to 1$ as $\sigma_1^{(m)}\to +\infty$, 
which is obvious.
We assume that the theorem is true for
$r-1$, and prove the theorem by induction.    We use the following harmonic product formula.
\begin{align}\label{y-7}
&\zeta(s_1)\zeta_{r-1}(s_2,\ldots,s_r)\\
&\;=\sum_{n_1=1}^{\infty}n_1^{-s_1}\sum_{n_2<\cdots<n_r}n_2^{-s_2}\cdots n_r^{-s_r}\notag\\
&\;=\sum_{n_1<n_2<\cdots<n_r}+\sum_{n_1=n_2<n_3<\cdots<n_r}+\sum_{n_2<n_1<n_3<\cdots<n_r}\notag\\
&\quad+\sum_{n_2<n_1=n_3<\cdots<n_r}+\cdots+\sum_{n_2<\cdots<n_{r-1}<n_1<n_r}\notag\\
&\quad+\sum_{n_2<\cdots<n_{r-1}<n_1=n_r}+\sum_{n_2<\cdots<n_{r-1}<n_r<n_1}\notag\\
&\;=\zeta_r(s_1,s_2,\ldots,s_r)+\zeta_{r-1}(s_1+s_2,s_3,\ldots,s_r)
+\zeta_r(s_2,s_1,s_3,\ldots,s_r)\notag\\
&\quad+\zeta_{r-1}(s_2,s_1+s_3,\ldots,s_r)+\cdots+\zeta_r(s_2,\ldots,s_{r-1},s_1,s_r)\notag\\
&\quad+\zeta_{r-1}(s_2,\ldots,s_{r-1},s_1+s_r)+\zeta_r(s_2,\ldots,s_{r-1},s_r,s_1),\notag
\end{align}
which is first valid in the region $\Re s_k>1$ ($1\leq k\leq r$), but then is valid in the
whole space by meromorphic continuation.

Let $s_k=s_k^{(m)}$ ($1\leq k\leq r$) in \eqref{y-7}, and
consider the situation when $\sigma_1^{(m)}\to +\infty$ and
$(s_2^{(m)},\ldots,s_r^{(m)})$ remains in $D_{r-1}$.

The singularities appearing in the formula \eqref{y-7} consist of two types: included
in $S_{1,r}$, or some hyperplane whose defining equation includes $s_1$. 
Therefore by \eqref{sing_condition2} we find that those singularities are irrelevant
during our process $\sigma_1^{(m)}\to +\infty$.

Therefore $\zeta_{r-1}(s_2^{(m)},\ldots,s_r^{(m)})$ remains bounded.
Noting $\zeta(s_1^{(m)})\to 1$ we see that
there exists a large $M_1=M_1(\varepsilon)$ for which
\begin{align}\label{y-7-1}
|\zeta(s_1^{(m)})\zeta_{r-1}(s_2^{(m)},\ldots,s_r^{(m)})-\zeta_{r-1}(s_2^{(m)},\ldots,s_r^{(m)})|
<\varepsilon/3
\end{align}
holds for any $m\geq M_1$.

On the right-hand side of \eqref{y-7},
denote by $G(s_1,\ldots,s_r)$ the sum of all the terms, except for the first two terms.
Then $G(s_1^{(m)},\ldots,s_r^{(m)})$ tends to 0, in view of Theorem \ref{th-y-0}.
That is, there exists $M_2=M_2(\varepsilon)$ for which
\begin{align}\label{y-7-2}
|G(s_1^{(m)},\ldots,s_r^{(m)})|<\varepsilon/3
\end{align}
holds for any $m\geq M_2$.

We use the induction assumption to treat the second term on the right-hand side.
We can find $M_3=M_3(\varepsilon)$ for which
\begin{align}\label{y-7-3}
|\zeta_{r-1}(s_1^{(m)}+s_2^{(m)},s_3^{(m)},\ldots,s_r^{(m)})-F_{r-2}(s_3^{(m)},\ldots,s_r^{(m)})|
<\varepsilon/3
\end{align}
for any $m\geq M_3$.

Substituting \eqref{y-7-1}, \eqref{y-7-2} and \eqref{y-7-3} into \eqref{y-7} we find that
\begin{align*}
&|\zeta_r(s_1^{(m)},\ldots,s_r^{(m)})-
(\zeta_{r-1}(s_2^{(m)},\ldots,s_r^{(m)})
-F_{r-2}(s_3^{(m)},\ldots,s_r^{(m)}))|\\
&\qquad<\varepsilon/3+\varepsilon/3+\varepsilon/3,
\end{align*}
that is,
\begin{align*}
|\zeta_r(s_1^{(m)},\ldots,s_r^{(m)})-F_{r-1}(s_2^{(m)},\ldots,s_r^{(m)})|<\varepsilon.
\end{align*}    
The uniformity of the convergence follows from the assertion on the uniformity in Theorem
\ref{th-y-0}.
This completes the proof of Theorem \ref{th-y-1}, 
and hence the
part (i) of Theorem \ref{th-y-2}.
\end{proof}

\begin{rem}
(A generalization of Remark \ref{rem_trivial})
If $\Re s_r>1$, then \eqref{y-2} is valid.    In this case, when $\sigma_1\to +\infty$ and 
$(s_2^{(m)},\ldots,s_r^{(m)})$ tends to $(\rho_2,\ldots,\rho_r)$,
from \eqref{y-2} it is immediate that 
$\zeta_r(s_1^{(m)},\ldots,s_r^{(m)})$ tends to
\begin{align}\label{y-8}
\sum_{n_2=1}^{\infty}\cdots\sum_{n_r=1}^{\infty}(1+n_2)^{-\rho_2}\cdots(1+n_2+\cdots+n_r)^{-\rho_r}.
\end{align}
Therefore in the region $\Re s_r>1$, the two expressions $F_{r-1}(\rho_2,\ldots,\rho_r)$ 
and \eqref{y-8} should be
equal to each other.
In fact, putting $1+n_2=m_2$ in \eqref{y-8}, we see that \eqref{y-8} is equal to
\begin{align*}
&\sum_{m_2=2}^{\infty}\sum_{n_3=1}^{\infty}\cdots\sum_{n_r=1}^{\infty}
m_2^{-\rho_2}(m_2+n_3)^{-\rho_3}\cdots(m_2+n_3+\cdots+n_r)^{-\rho_r}\\
&=\zeta_{r-1}(\rho_2,\rho_3,\ldots,\rho_r)-\sum_{n_3=1}^{\infty}\cdots\sum_{n_r=1}^{\infty}
(1+n_3)^{-\rho_3}\cdots(1+n_3+\cdots+n_r)^{-\rho_r}.
\end{align*}
Next we put $1+n_3=m_3$ and argue similarly to obtain that the second sum on the right-hand side
is
$$
=\zeta_{r-2}(\rho_3,\ldots,\rho_r)-\sum_{n_4=1}^{\infty}\cdots\sum_{n_r=1}^{\infty}
(1+n_4)^{-\rho_4}\cdots(1+n_4+\cdots+n_r)^{-\rho_r}.
$$
Repeating this argument, we arrive at the expression $F_{r-1}(\rho_2,\ldots,\rho_r)$.
\end{rem}

Lastly we prove the part (ii) of Theorem \ref{th-y-2}.
Let $(\rho_2,\ldots,\rho_r)$ be any solution of \eqref{y-1}.
We fix $\rho_2,\ldots,\rho_{r-1}$, and regard 
$F_{r-1}(\rho_2,\ldots,\rho_{r-1},s_r)$ as a function of one complex variable $s_r$
(and we denote it by $F(s_r)$ for brevity).
Then $F(\rho_r)=0$, and since any zero of function of one complex variable is isolated, we can find
a small positive $\varepsilon_0$, with the properties that $F(s_r)$ is holomorphic in
$|s_r-\rho_r|<\varepsilon_0$ and $s_r=\rho_r$ is the only zero point in this region.
Choose $0<\varepsilon<\varepsilon_0$, and put
\begin{align}\label{y-8-5}
m(\varepsilon)=\min_{|s_r-\rho_r|=\varepsilon}|F(s_r)|>0.
\end{align}

Now, put $s_k=\rho_k$ ($2\leq k\leq r-1$) in \eqref{y-7}.    Also assume that
$|s_r-\rho_r|=\varepsilon$.    Then $|\zeta_{r-1}(\rho_2,\ldots,\rho_{r-1},s_r)|$ is bounded, 
so we have
\begin{align}\label{y-9}
|\zeta(s_1)\zeta_{r-1}(\rho_2,\ldots,\rho_{r-1},s_r)
-\zeta_{r-1}(\rho_2,\ldots,\rho_{r-1},s_r)|<m(\varepsilon)/3
\end{align}
if $\sigma_1>M_1(\varepsilon)$ with a sufficiently large $M_1(\varepsilon)$.

By Theorem \ref{th-y-0}, there exists
$M_2=M_2(\varepsilon)$ for which 
\begin{align}\label{y-10}
|G(s_1,\rho_2,\ldots,\rho_{r-1},s_r)|<m(\varepsilon)/3
\end{align}
for $\sigma_1>M_2(\varepsilon)$.

Lastly, by Theorem \ref{th-y-1} we see that there exists
$M_3=M_3(\varepsilon)$ for which
\begin{align}\label{y-11}
&\Bigl|\zeta_{r-1}(s_1+\rho_2,\rho_3,\ldots,\rho_{r-1},s_r)\Bigl.\\
&\bigl.-(\sum_{j=1}^{r-2}(-1)^{r-j}\zeta_j(\rho_{r-j+1},\ldots,\rho_{r-1},s_r)-(-1)^{r-1})\Bigr|
<m(\varepsilon)/3\notag
\end{align}
for $\sigma_1>M_3(\varepsilon)$.

Combining \eqref{y-7}, \eqref{y-9}, \eqref{y-10} and \eqref{y-11}, and noting
\begin{align*}
&\zeta_{r-1}(\rho_2,\ldots,\rho_{r-1},s_r)-
(\sum_{j=1}^{r-2}(-1)^{r-j}\zeta_j(\rho_{r-j+1},\ldots,\rho_{r-1},s_r)-(-1)^{r-1})\\
&\quad=F_{r-1}(\rho_2,\ldots,\rho_{r-1},s_r)=F(s_r),
\end{align*}
we obtain
\begin{align*}
&|\zeta_r(s_1,\rho_2,\ldots,\rho_{r-1},s_r)
-F(s_r)|\\
&\qquad<m(\varepsilon)/3+m(\varepsilon)/3+m(\varepsilon)/3=m(\varepsilon)
\end{align*}
when $\sigma_1>\max\{M_1(\varepsilon),M_2(\varepsilon),M_3(\varepsilon)\}$.
We fix such an $s_1$.
From the above inequality and \eqref{y-8-5} we obtain
$$
|\zeta_r(s_1,\rho_2,\ldots,\rho_{r-1},s_r)
-F(s_r)|
<|F(s_r)|
$$
for all $s_r$ satisfying $|s_r-\rho_r|=\varepsilon$.
Therefore by Rouch{\'e}'s theorem we find that 
the number of zeros of $\zeta_r(s_1,\rho_2,\ldots,\rho_{r-1},s_r)$ (as a function in $s_r$)
in the region $|s_r-\rho_r|<\varepsilon$ is equal to the number of zeros of $F(s_r)$ in the
same region, but the latter is 1.   This completes the proof of Theorem \ref{th-y-2} (ii).

\bigskip
\noindent
\parbox[t]{.48\textwidth}{
Kohji Matsumoto\\
Graduate School of Mathematics\\
Nagoya University\\
Chikusa-ku, Nagoya 464-8602, Japan\\
kohjimat@math.nagoya-u.ac.jp } \hfill
\parbox[t]{.48\textwidth}{
Mayumi Sh{\=o}ji\\
Department of Mathematical and Physical Sciences\\
Japan Women's University\\
Mejirodai, Bunkyoku, Tokyo 112-8681, Japan\\
shoji@fc.jwu.ac.jp }

\end{document}